\documentclass[a4paper, 10pt]{amsart}
 \usepackage{amsmath,amscd,latexsym,verbatim,amssymb}
 \usepackage[latin1]{inputenc}
 \usepackage{times} 
\usepackage[all]{xy}

 \newcommand{\resp}{{\it resp.} }
\newcommand{\cf}{{\it cf.} }
\newcommand{\ie}{{\it i.e.} }
\newcommand{\eg}{{\it e.g.} }
\newcommand{\loccit}{{\it loc. cit.} }

\renewcommand{\qed}{\hfill$\Box$\medskip}

\newcommand{\sB}{\mathcal{B}}

\newcommand{\sK}{\mathcal{K}}

\newcommand{\Q}{\mathbf{Q}}
\newcommand{\Z}{\mathbf{Z}}

\newcommand{\N}{\mathbf{N}}
\newcommand{\R}{\mathbf{R}}
 \newcommand{\inj}{\hookrightarrow}
 \newcommand{\e}{\frac{1}{p^\infty}}
 \newcommand{\f}{-\frac{1}{p^\infty}}
\newcommand{\s}{\sm o}

\renewcommand{\ker}{\rm{ker}\,}

\newcommand{\Spec}{\rm{Spec}\,}

\renewcommand{\epsilon}{\varepsilon}
\font\sm=cmr10 at9pt

\usepackage{hyperref}
\usepackage{color}
\swapnumbers
 \newtheorem{thm}{Th\'eor\`eme}[subsection]
 \newtheorem{lemma}[thm]{Lemme}

\newtheorem{prop}[thm]{Proposition}

\newtheorem{cor}[thm]{Corollaire}
\newtheorem{conj}[thm]{Conjecture}

\newtheorem{defn}[thm]{Definition}

 at13pt
\font\sm=cmr9 at6pt

\begin{document}

 \title[Conjecture du facteur direct]{La conjecture du facteur direct}

\author{Yves
Andr\'e}

\address{Institut de Math\'ematiques de Jussieu\\  4 place Jussieu, 75005
Paris\\France.}
\email{yves.andre@imj-prg.fr}
   \date{23 ao\^ut 2016}
\keywords{{Direct summand conjecture, big Cohen-Macaulay algebra, perfectoid algebra, purity}} \subjclass{13D22, 13H05, 14G20}

   \bigskip 
   
   \bigskip 
    
  \medskip \begin{abstract} M. Hochster a conjectur\'e que pour toute extension finie $S$ d'un anneau commutatif r\'egulier $R$, la suite exacte de $R$-modules $0\to R \to S \to S/R\to 0$ est scind\'ee. En nous appuyant sur sa r\'eduction au cas d'un anneau local r\'egulier $R$ complet non ramifi\'e d'in\'egale caract\'eristique, nous proposons une d\'emonstration de cette conjecture dans le contexte de la th\'eorie perfecto\"{\i}de de P. Scholze. Les deux ingr\'edients-cl\'e sont le {``lemme d'Abhyankar"} perfecto\"{\i}de et l'analyse des extensions kumm\'eriennes de $R$ par une technique d'\'epaississement sur des voisinages tubulaires. 
    
  \bigskip
\noindent{\sm{ABSTRACT.}} M. Hochster conjectured that any finite extension of a regular commutative ring splits as a module. Building on his reduction to the case of an unramified complete regular local ring $R$ of mixed characteristic, we propose a proof in the framework of P. Scholze's perfectoid theory.  The main ingredients are the perfectoid ``Abhyankar lemma" and an analysis of Kummer extensions of $R$ by a thickening technique.
  \end{abstract}

     \bigskip 
        \bigskip  
   \maketitle
  \let\languagename\relax

\tableofcontents

 \newpage
 
     \begin{sloppypar}

      \section*{Introduction.} 
 
  \subsection{}\label{fd}   La conjecture du facteur direct, publi\'ee par M. Hochster en 1973 \cite{H1}, est l'\'enonc\'e suivant, d'apparence \'el\'ementaire:

\begin{conj}\label{conj} Soit $R$ un anneau commutatif noeth\'erien \emph{r\'egulier}. Alors pour toute $R$-alg\`ebre commutative fid\`ele et \emph{finie} $S$, l'inclusion $R\inj S$ admet une \emph{r\'etraction} $R$-lin\'eaire, \ie $R$ est facteur direct de $S$ en tant que $R$-module.  \end{conj} 

Il revient au m\^eme de dire que toute extension finie $R\inj S$ est {\it pure}, \ie universellement injective,
 ou encore que $S$ est un {g\'en\'erateur} de la cat\'egorie des $R$-modules.
  
 Cette conjecture occupe une place centrale dans l'\'echeveau des {``conjectures homologiques"}, issues des travaux de C. Peskine et L. Szpiro \cite{PS}, et qui structurent la vision sous-jacente \`a  bien des travaux d'alg\`ebre commutative depuis une quarantaine d'ann\'ees \cite{H4}. 
Le r\'esultat suivant, qui synth\'etise les travaux de plusieurs auteurs, donne un aper\c cu de son caract\`ere prot\'eiforme.  
 
 \smallskip {\it $\;\;\;\;\;\;\;\;$ Les \'enonc\'es suivants sont \'equivalents\footnote{Une suite s\'ecante maximale est un syst\`eme de $(\dim R)$ g\'en\'erateurs d'un id\'eal dont le radical est l'id\'eal maximal (``system of parameters" en anglais). Le rang d'un module de dimension projective finie (sur un anneau local) est la somme altern\'ee des rangs dans une r\'esolution libre. 
 
 Pour l'\'equivalence $(1)\Leftrightarrow (3)$, voir \cite[th. 6.1]{H2}.  
  L'implication $(1)\Rightarrow (2)$ est \'el\'ementaire: la puret\'e de l'extension $S$ entra\^{\i}ne que $R/I \to S/IS$ est injectif; pour la r\'eciproque, voir \cite{D}. Pour $(3)\Rightarrow (4)$, voir \cite{H2}, et \cite{Du} pour la r\'eciproque. L'implication $(4)\Rightarrow (5)$, implicite dans \cite{EG}, est explicit\'ee dans \cite{H2}. }:
   \begin{enumerate} 
 \item La conjecture du facteur direct vaut pour tout anneau r\'egulier.
 \item Pour tout id\'eal $I$ d'un anneau r\'egulier $R$ et toute extension {finie} $S$, $I = R \cap IS$. 
 \item Pour tout anneau local (noeth\'erien) $R$,
  toute suite s\'ecante maximale $(x_1,  \ldots, x_d) \,$ et tout couple $(m,n)\in \N^2$, le mon\^ome $(x_1\cdots x_d)^m$ est dans l'id\'eal engendr\'e par $\,x_1^n, \ldots , x_d^n\,$ si et seulement si $\, m\geq n$.
 \item Pour tout anneau local $R$ d'id\'eal maximal $\frak m$, et tout complexe de modules libres de type fini $0\to F_d \to \cdots \to F_0\to 0$ tel que $H_i(F_\bullet)$ soit de longueur finie pour $i>0$ et que $H_0(F_\bullet)\setminus \frak m H_0(F_\bullet)$ ait de la torsion $\frak m$-primaire,
   on a $\dim\, R\leq d$.
  
  \medskip Ils impliquent la conjecture des syzygies: \medskip
  
  \item Tout $k$-i\`eme module de syzygies $M$ (d'un module de type fini sur un anneau local), qui est de dimension projective finie mais non libre, est de rang $\geq k$.  
   \end{enumerate} }

  \subsection{}\label{t1}  La conjecture du facteur direct est un probl\`eme local sur $R$ mais pas sur l'extension $S$ (probl\`eme de recollement des r\'etractions). Si le degr\'e de l'extension est inversible, il est facile de construire une r\'etraction \`a l'aide d'une trace divis\'ee, ce qui \'etablit la conjecture lorsque $R$ contient $\Q$. C'est encore facile si l'extension finie est plate  (donc fid\`element plate, de sorte que la suite exacte $0\to R \to S \to S/R\to 0$ se scinde), ce qui \'etablit la conjecture en dimension $\leq 2$, puisque toute extension finie normale de $R$ est alors plate. Le cas beaucoup plus ardu de la dimension $3$ a \'et\'e r\'esolu par R. Heitman \cite{He}.
  
  Par ailleurs, sans l'hypoth\`ese de r\'egularit\'e, il est ais\'e de trouver des contre-exemples:
    l'id\'eal $(x+y)$ de $R= K[x,y]/(xy)$ n'est le contract\'e d'aucun id\'eal du normalis\'e $K[x]\times K[y]$; la normalit\'e ne suffirait pas, d'ailleurs \cite[ex. 1]{H1}.
 
 \smallskip Hochster a d\'emontr\'e la conjecture en caract\'eristique $p>0\,$  
 (voir \ref{astH} ci-dessous), et a {\it ramen\'e le cas g\'en\'eral au cas d'un anneau local complet non ramifi\'e d'in\'egale caract\'eristique  $(0,p)$ de corps r\'esiduel ${k}$ parfait}, c'est-\`a-dire, en vertu du th\'eor\`eme de structure de Cohen, au cas d'un anneau de s\'eries formelles $W({k})[[T_1,\ldots, T_n]]$ \`a coefficients dans l'anneau de Witt de ${k}\,$ \cite[th. 6.1]{H2}.

 \newpage \medskip L'objectif de cet article est de la d\'emontrer en g\'en\'eral, via cette r\'eduction:
  
\begin{thm}\label{T1}  La conjecture du facteur direct est vraie pour $W({k})[[T_1,\ldots, T_n]]$  (donc aussi pour tout anneau r\'egulier $R$). \end{thm} 
 
  Nous ferons usage de techniques {``transcendantes"} issues de la th\'eorie de Hodge $p$-adique,  quittant d\'elib\'er\'ement le monde noeth\'erien o\`u l'alg\`ebre commutative est diserte, pour le non-noeth\'erien o\`u elle ann\^one.

  \subsection{}\label{carp} Expliquons la strat\'egie de la preuve dans le cas analogue, mais beaucoup plus simple, o\`u 
 $\,R=  {k}[[T_0,\ldots, T_n]]\,$  et o\`u l'extension finie $\,S\,$ de $\,R\,$ est int\`egre et munie d'un groupe fini $\,G\,$ d'automorphismes tel que $\,S^G=R$. 
  
  La trace $\,{\rm{tr}}:\, S\to R\, $ donn\'ee par la somme des conjugu\'es est non nulle, car l'extension des corps de fractions est galoisienne; n\'eanmoins $\,{\rm{tr}}\,$ n'est pas surjective en g\'en\'eral,  donc est impropre \`a fournir une r\'etraction de $\,R\inj S\,$ si $\,p\,$ divise l'ordre de $G$.
 
 La situation s'am\'eliore en passant aux cl\^otures parfaites: l'extension des corps de fractions demeure galoisienne de groupe $G$, mais l'id\'eal non nul ${\rm{tr}}(S^{1/ p^\infty})$ de $R^{1/ p^\infty}$ devient radiciel: ${\rm{tr}} (S^{1/ p^\infty})= ({\rm{tr}} (S^{1/ p^\infty}))^{1/ p^\infty}$. En particulier, ${\rm{tr}} (S^{1/ p^\infty})$ n'est pas contenu dans l'id\'eal $(T_0, T_1, \ldots ,  T_n)R^{1/ p^\infty}$.
  Comme $R^{1/ p^\infty}$ est libre sur $\,R\,$ de base $(T_0^{m_0}\cdots  T_n^{m_n})_{{\underline m}\in (\Z[\frac{1}{p}]\cap [0,1[)^{n+1}}$,  on en d\'eduit l'existence d'un \'el\'ement $\,s \in S^{1/ p^\infty}\,$ et d'un indice $\,\underline m\,$ tels que la projection de  ${\rm{tr}} (s)$ sur le facteur $\,R\,$ index\'e par $\,\underline m\,$ soit \'egal \`a $1$.  En composant les applications $R$-lin\'eaires $\,S\to S^{1/ p^\infty} \stackrel{s'\mapsto {\rm{tr}}(ss') }{\to } R^{1/ p^\infty} \stackrel{{\rm{pr}}_{\underline{m}}}{\to} R$, on obtient la r\'etraction cherch\'ee.

   \subsection{} En in\'egale caract\'eristique, o\`u $W({k})$  se substitue \`a  ${k}[[T_0]]$, la th\'eorie de Hodge $p$-adique sugg\`ere une approche analogue, en rempla\c cant cl\^otures radicielles par {\it extensions profond\'ement ramifi\'ees}. Dans le cas de $A:= W({k})[[T_1,\ldots, T_n]]$, on peut par exemple consid\'erer $ \hat A_\infty^{\s} := W(k[[T_1^{\e},\ldots, T_n^{\e}]])\hat\otimes_{W(k)} \hat K_\infty^{\s}$, o\`u $\hat K_\infty^{\s}$ est le compl\'et\'e de la $\Z_p$-extension cyclotomique $K_\infty^{\s}= W(k)[\zeta_{p^\infty}]$ de $W({k})$. 
     
 \smallskip  Cette id\'ee a d\'ej\`a \'et\'e explor\'ee par plusieurs auteurs, dont P. Roberts, puis K. Shimomoto, B. Bhatt, O. Gabber et L. Ramero; elle a notamment permis \`a Bhatt  \cite{Bh} de prouver la conjecture du facteur direct dans le cas o\`u $B[\frac{1}{p}]$ est \'etale sur $A[\frac{1}{p}]$: dans ce cas, le th\'eor\`eme de {``presque-puret\'e"} de Faltings implique en effet que l'anneau des entiers $(B\otimes_A  \hat A_\infty)^{\s}$ est presque pur sur $\hat A_\infty^{\s}$, et un argument noeth\'erien permet de passer de l\`a \`a la puret\'e sur $A$. L'argument s'\'etend d'ailleurs au cas o\`u $B[\frac{1}{p}]$  n'est ramifi\'e qu'au-dessus de $T_1 \cdots T_n =0$.
 
   \subsection{}  Pour le cas g\'en\'eral, le th\'eor\`eme de Faltings s'av\`ere insuffisant. Nous le remplacerons par le {``lemme d'Abhyankar"} perfecto\"{\i}de de \cite{A}, qui permet de traiter le cas o\`u $B[\frac{1}{p}]$ est ramifi\'e sur $A[\frac{1}{p}]$ le long d'un discriminant $g\in A $ quelconque. 
   
  Ce r\'esultat affirme entre autre que {\it quitte \`a adjoindre les racines $p^\infty$-i\`emes de $g$ et prendre une fermeture [compl\`etement] int\'egrale}, {\it l'extension des anneaux d'entiers devient {``presque"} presque \'etale finie modulo toute puissance de $p$ - {``presque"}} \'etant entendu au sens o\`u l'on {``n\'eglige"} tout ce qui est annul\'e par $(\zeta_{p^j}-1)g^{\frac{1}{p^j}}$ pour tout $j$.  Plus pr\'ecis\'ement, la fermeture int\'egrale $\sB^{\s}$ de $g^{\f} \hat A_\infty\langle g^{\e}\rangle^{\s} $ dans $B\otimes_A  \hat A_\infty\langle g^{\e}\rangle[\frac{1}{g}]$ est {\it presque} \'etale finie et pure sur $ \hat A_\infty\langle g^{\e}\rangle^{\s}$ modulo toute puissance de $p$ (th. \ref{T4}).   
   
   La presque-alg\`ebre intervient ici dans un cadre in\'edit o\`u l'id\'eal idempotent n'est pas un id\'eal de valuation. On a par ailleurs remplac\'e $ \hat A_\infty^{\s}$ par $ \hat A_\infty\langle g^{\e}\rangle^{\s}$ et, pour conclure, on a besoin de propri\'et\'es de puret\'e ou de platitude de $ \hat A_\infty^{\s}\to \hat A_\infty\langle g^{\e}\rangle^{\s}$. 
   
    \subsection{} L'\'etude des extensions kumm\'eriennes $A[\zeta_{p^i}, g^{\frac{1}{p^i}}, {\frac{1}{p}}]^{\s}$ de $A$ est notoirement difficile: on sait que l'anneau des entiers de $A[g^{\frac{1}{p}}]$  (\resp $A[g^{\frac{1}{p^i}}]$) est pur sur $A$ mais, apr\`es adjonction d'une racine $p$-i\`eme de l'unit\'e, pas n\'ecessairement plat sur $A$ \cite{Ko} (\resp \cite{R}).   Notre approche consistera \`a travailler avec $\hat A_\infty\langle g^{\e}\rangle^{\s}$ directement, en l'{``\'epaississant "}, c'est-\`a-dire en la voyant comme colimite compl\'et\'ee-s\'epar\'ee d'alg\`ebres de fonctions born\'ees sur des voisinages tubulaires de $\,T=g\,$ dans le spectre analytique de $\hat A_\infty\langle T^{\e}\rangle $, et en exploitant le caract\`ere perfecto\"{\i}de de telles alg\`ebres. En r\'esum\'e, $ \hat A_\infty\langle g^{\e}\rangle^{\s}$ est {\it presque} fid\`element plate sur une certaine d\'ecompl\'etion $ A_\infty^{\s}$ de $ \hat A_\infty^{\s}$, {``presque "} \'etant entendu ici au sens o\`u l'on n\'eglige ce qui est annul\'e par l'id\'eal de valuation de $K_\infty^{\s}$ (th. \ref{T3}). On en d\'eduit assez facilement la puret\'e de $A \inj A[\zeta_{p^i}, g^{\frac{1}{p^i}}, {\frac{1}{p}}]^{\s}$ (cor. \ref{C2}).

    \subsection{}       
 Comme l'ont montr\'e les travaux de Heitman et surtout Hochster, la conjecture du facteur direct est proche de (et impliqu\'ee par) l'existence d'alg\`ebres de Cohen-Macaulay (non n\'ecessairement noeth\'eriennes) pour les anneaux locaux (\cf \eg \cite{H4}\cite{HH}). On peut combiner les r\'esultats esquiss\'es ci-dessus aux techniques de \cite{H3} pour \'etablir cette existence.
       
     \begin{thm}\label{T2} Pour tout anneau local noeth\'erien complet int\`egre $B$ d'in\'egale caract\'eristique $(0,p)$, il existe une $B$-alg\`ebre $C$ de Cohen-Macaulay, \ie telle que toute suite s\'ecante maximale de $B$ devient une suite r\'eguli\`ere dans $C$.  
     \end{thm}

  \medskip
   \bigskip
\begin{small} {\it Remerciements.} Ma vive reconnaissance va \`a Luisa Fiorot, qui m'a fait conna\^itre la conjecture du facteur direct fin 2012, et 
  m'a expliqu\'e son importance dans la probl\'ematique de la descente.  \end{small}

   \bigskip 
   
 \bigskip
 
 \bigskip

  \bigskip 
   
 \bigskip
 
 \bigskip 
 
    \section{Pr\'eliminaires.}
    
       \subsection{}    Commen\c cons par deux lemmes aux confins de l'alg\`ebre noeth\'erienne. 

       \begin{lemma}\label{L1} Soient $R$ un anneau noeth\'erien, ${\frak I}$ un id\'eal et $S$ une $R$-alg\`ebre plate (non n\'ecessairement noeth\'erienne). Alors le compl\'et\'e-s\'epar\'e ${\frak I}$-adique $\hat S$ de $S$ est plat sur $R$, et pour tout module $M$ de type fini sur $R$, le compl\'et\'e de $M_S$ s'identifie \`a $M_{\hat S}$.   Si ${\frak I}$ est contenu dans le radical de Jacobson de $R$ (ce qui est le cas si $R$ est $\frak I$-adiquement complet) 
          et si $S$ est fid\`element plat sur $R$,  alors $\hat S$ est fid\`element plat sur $R$.    \end{lemma} 
 
    \begin{proof}\footnote{ce lemme appara\^it sous diverses formes et avec diverses preuves dans la litt\'erature, par exemple \cite[th. 0.1]{Y}; la preuve qui suit est plus \'el\'ementaire que celle de \loccit. L'\'enonc\'e est \`a premi\`ere vue surprenant dans le cas o\`u $S$ n'est pas s\'epar\'e; il ne dit rien d'ailleurs sur le s\'epar\'e de $S$.}      Consid\'erons une suite exacte courte $0\to M_1\to M_2 \to M_3  \to 0$  de  $R$-modules de type fini. D'apr\`es Artin-Rees, il existe $m$ tel que pour tout $n\geq m$, en posant $N_1 = M_1 \cap \frak I^m M_2$, la suite $0\to M_1/{\frak I}^{n-m} N_1 \to M_2/ {\frak I}^n M_2 \to M_3  /{\frak I}^n M_3 \to 0$ soit exacte. On obtient encore une suite exacte en tensorisant avec la $R$-alg\`ebre plate $S$, puis, d'apr\`es Mittag-Leffler, en passant \`a la limite sur $n$. Puisque ${\frak I}^{n-m}N_1$ est coinc\'e entre ${\frak I}^n M_1$ et ${\frak I}^{n-m}M_1$, on obtient une suite exacte courte  $0\to \widehat{M_{1S}} \to \widehat{M_{2S}} \to \widehat{M_{3S}}  \to 0$. 
   
Pour tout $R$-module de type fini $M$, le morphisme
  $M_{\hat S} \to \widehat{M_{ S}} $ est surjectif, 
    et on d\'eduit de ce qui pr\'ec\`ede, en prenant une pr\'esentation de $M$, qu'il est en fait bijectif. 
   On d\'eduit de l\`a et de la suite exacte pr\'ec\'edente que  $\hat S$  est plat sur $R $.  
   
    Supposons ensuite qu'on ait $M_{\hat S} =0$. Alors  $(M/{\frak I}M)\otimes_R S =M\otimes_R (\hat S/ \frak I\hat S) = 0$.  Si $S$ est fid\`element plat, on a $M= {\frak I}M$. Si ${\frak I}$ est contenu dans le radical, on a $M=0$ d'apr\`es Nakayama, et on conclut que $\hat S$  est fid\`element plat sur $R $, ce qui ach\`eve la preuve du lemme. \end{proof}

   \begin{lemma}\label{L2} Soient $R$ est un anneau noeth\'erien, ${\frak J}$ un id\'eal, $S$ une $R$-alg\`ebre fid\`element plate (ou pure), et $\frak K $ un id\'eal idempotent de $S$ contenu dans $ \frak J S$.    
      Alors pour tout $s\in  R \cap \frak K  $, il existe $r\in \frak J$ tel que $(1-r) s =0$. En particulier, si $R$ est local et $R\cap  \frak K   \neq 0$, alors $\frak J = R$.  
     \end{lemma}
     
   \begin{proof} Pour tout $n\geq 1$, $\frak K= \frak K^n$ est contenu dans $(\frak J S)^n =  \frak J^n S$. Par platitude fid\`ele (ou puret\'e), on a $R \cap (\frak J S)^n   = R \cap ({\frak J}^{n}\otimes S) = {\frak J}^{ n}$  \cite[I, \S 2, n. 6, prop. 7]{B}. Donc $R\cap \frak K$ est contenu dans  $\cap\,{\frak J}^{ n}$. On conclut par le lemme de Krull. \end{proof}
   
   \subsection{}   Voici quelques rappels et compl\'ements sur la {\it localisation de Weierstrass} des alg\`ebres de Banach uniformes (\ie dont la norme est \'equivalente \`a la norme spectrale associ\'ee \cite[2.2]{A}). 
   
 Soient $\sK$ un corps complet pour une valuation non-triviale, $\lambda $ un \'el\'ement non nul de l'anneau de valuation $\sK^{\s}$. Soient  $\sB$ une $\sK$-alg\`ebre de Banach uniforme, et $f$ un \'el\'ement de la boule unit\'e $\sB_{\leq 1}$. Alors $\sB\{ \frac{f}{\lambda}\}$ est d\'efinie comme le quotient de $\sB\langle U \rangle$ par l'adh\'erence de l'id\'eal engendr\'e par $\lambda U - f$; c'est une alg\`ebre de Banach pour la norme quotient (non n\'ecessairement uniforme). En fait, cet id\'eal est {\it ferm\'e}  \cite[prop. 2.3]{M}, de sorte que \begin{equation}\label{m1}\sB\{ \frac{f}{\lambda}\}= \sB\langle U \rangle/ (\lambda U - f).\end{equation} 
   Si $f$ est non-diviseur de z\'ero dans $\sB$, $\lambda U- f$ l'est aussi dans $\sB\langle U \rangle$.

     \subsubsection{} 
  Supposons que $\vert \sK\vert$ soit dense dans $\vert \sB\vert$ (ce qui est le cas si la valuation est non discr\`ete). Pour tout \'el\'ement $\varpi\in \sK^{\s\s}$, la topologie de $\sB_{\leq 1}$ est la topologie $\varpi$-adique \cite[sor. 2.3.1]{A}. 
  
  Supposons que {\it la multiplication par $f$ soit isom\'etrique} dans $\sB$ (ce qui est en particulier le cas si la norme de $\sB$ est multiplicative et $\vert f\vert = 1$). Alors le quotient  $  \sB_{\leq 1}\langle U \rangle/ (\lambda U - f)$ est sans $\varpi$-torsion: en effet, il suffit de voir que tout \'el\'ement annul\'e par $\lambda$ est nul; or si $\sum b_m U^m$ rel\`eve un tel \'el\'ement, l'\'equation $\lambda (\sum b_m U^m)=    (\lambda U-  f )\sum_0^\infty a_m U^m, \;\; a_m \in \sB_{\leq 1},$ implique que $\lambda$ divise tous les $ a_m$ puisque la multiplication par $f$ est injective modulo $\lambda$. On a donc $ (\lambda U - f)  \sB_{\leq 1} =  ((\lambda U - f)  \sB)_{\leq 1}$, et il en d\'ecoule que 
   \begin{equation}\label{m2}\sB\{ \frac{f}{\lambda}\}_{\leq 1}= \sB_{\leq 1}\langle U \rangle/ (\lambda U - f)\end{equation}
  si la valuation de $\sK$ et discr\`ete, et
    \begin{equation}\label{m2'}\sB\{ \frac{f}{\lambda}\}_{\leq 1}= (\sB_{\leq 1}\langle U \rangle/ (\lambda U - f))^a_\ast\end{equation}
sinon, en utilisant la notation $(\,)^a_\ast$ de la presque-alg\`ebre dans le cadre $(\sK^{\s}, \sK^{\s\s})$, qui se traduit ici par $\displaystyle\frak B^a_\ast := \bigcap_{\eta\in \sK^{\s\s}}\, \eta^{-1}\frak B\;$ \cite[sor. 2.3.1]{A}.

 \smallskip  La formule \eqref{m1} montre par ailleurs que $ \sB_{\leq 1}\langle U \rangle/ (\lambda U - f)$ est $\varpi$-adiquement s\'epar\'e, \ie $ (\lambda U - f)$ est ferm\'e dans $ \sB_{\leq 1}\langle U \rangle$. On en d\'eduit l'\'egalit\'e
  \begin{equation}\label{m3}\sB_{\leq 1}\langle U \rangle/ (\lambda U - f) = \widehat{\sB_{\leq 1}[\frac{f}{\lambda}]}, \end{equation} 
avec le compl\'et\'e $\varpi$-adique de $\sB_{\leq 1}[\frac{f}{\lambda}]\subset \sB$. En effet, le m\^eme argument que ci-dessus montre que $ \sB_{\leq 1}[U] / (\lambda U - f) $ est sans $\varpi$-torsion, de sorte que la suite
 \begin{equation} 0 \to (\lambda U - f) \sB_{\leq 1}[U] \to  \sB_{\leq 1}[U] \to \sB_{\leq 1}[\frac{f}{\lambda}] \to 0  \end{equation} est exacte. Elle induit comme d'habitude par compl\'etion
 une suite 
 \begin{equation}\label{m4} 0\to \widehat{(\lambda U - f) \sB_{\leq 1}[U]}   \to  \sB_{\leq 1}\langle U\rangle \to \widehat{\sB_{\leq 1}[\frac{f}{\lambda}]} \to 0 \end{equation}  exacte \`a droite et o\`u l'image de $ (\lambda U - f) \sB_{\leq 1}\langle U\rangle$ est dense dans le noyau de $ \sB_{\leq 1}\langle U\rangle \to \widehat{\sB_{\leq 1}[\frac{f}{\lambda}]}$; or on a $\widehat{(\lambda U - f) \sB_{\leq 1}[U]} = (\lambda U - f) \sB_{\leq 1}\langle U\rangle$ et la suite \eqref{m4}  est exacte \`a gauche; et puisque $(\lambda U - f) \sB_{\leq 1}\langle U\rangle$ est ferm\'e dans $\sB_{\leq 1}\langle U\rangle$, on conclut que la suite \eqref{m4} est exacte.   
 
  \smallskip  L'alg\`ebre $\sB_{\leq 1}\langle U \rangle/ (\lambda U - f) =  \widehat{\sB_{\leq 1}[\frac{f}{\lambda}]}$ ne change pas, \`a isomorphisme pr\`es, si l'on change $f$ en $f'$ tel que $\vert f - f'\vert \leq \vert\lambda\vert$ (\resp $\lambda$ en $\lambda'$ tel que $\vert \lambda\vert = \vert \lambda'\vert$), l'isomorphisme \'etant induit par $U \mapsto \frac{\lambda}{\lambda'} (U + h)$ o\`u $h = \lambda^{-1}(f'-f)\in \sB_{\leq 1}$.

 \section{Extensions {``kumm\'eriennes "} de $W({k})[[T_1,\ldots, T_n]]$.}

\subsection{} Pour tout corps $p$-adique $K$, on note $K^{\s}$ l'anneau de valuation et $K^{\s\s}$ l'id\'eal de valuation. 

 Soient $k$ un corps parfait de caract\'eristique $p$, $K_0 $ le corps des fractions de l'anneau des vecteurs de Witt $W(k)$. 
  Consid\'erons la tour cyclotomique $K_\infty = \cup K_j$ avec $K_j = K_0[\zeta_{p^j}]$. Le compl\'et\'e $p$-adique $\hat K_\infty $ est alors un corps perfecto\"{\i}de: l'endomorphisme de Frobenius de $\hat K_\infty^{\s}/p$ est surjectif $\,$ (voir \cite[\S 3.1, 3.2]{A} pour les d\'efinitions et r\'esultats de base concernant les corps et alg\`ebres perfecto\"{\i}des). 
  
\subsection{}  Comme dans l'introduction, posons \begin{equation} A  := W({k})[[T_1,\ldots, T_n]] \end{equation} 
et fixons un \'el\'ement non nul $g\in A$.  

 Posons $\,K^{\s}_j[[T_{\leq n}^{\frac{1}{p^j}}]][g^{\frac{1}{p^k}}] := K^{\s}_j[[T_1^{\frac{1}{p^j}},\ldots, T_n^{\frac{1}{p^j}}]][T]/(T^{p^k} - g) \,$ et
 \begin{equation} A_{jk} = K^{\s}_j[[T_{\leq n}^{\frac{1}{p^j}}]][g^{\frac{1}{p^k}}][\frac{1}{p}]. \end{equation} 
 Lorsque $(j,k)$ varie dans $ \N^2$, on obtient un double syst\`eme inductif de $K_0$-alg\`ebres, dont les morphismes de transition sont les inclusions naturelles. On permet la valeur $\infty$ pour l'un des indices (ou les deux), en prenant la r\'eunion index\'ee par les valeurs finies de cet indice.

 \subsection{}   On note $A_{jk}^{\s}$ la fermeture int\'egrale de $A$ dans $A_{jk} $. Cette $K_j^{\s}$-alg\`ebre contient $K^{\s}_j[[T_{\leq n}^{\frac{1}{p^j}}]][g^{\frac{1}{p^k}}]$ et v\'erifie $A_{jk} = A_{jk}^{\s}[\frac{1}{p}]$. Elle est noeth\'erienne et $p$-adiquement compl\`ete si $(j,k) \in \N^2$. Si $j=\infty$, elle est r\'eunion croissante d'alg\`ebres $A_{j'k'}^{\s}$ qui sont noeth\'eriennes, int\'egralement ferm\'ees dans $A_{j'k'}$, et finies les unes sur les autres, donc elle est compl\`etement int\'egralement ferm\'ee dans $A_{\infty k}$ (\ie tout \'el\'ement de $A_{\infty k}$ dont les puissances sont contenues dans un sous-$A_{\infty k}^{\s}$-module de type fini appartient \`a $A_{\infty k}^{\s}$). On note $\hat A_{\infty k}^{\s}$ le compl\'et\'e $p$-adique de $A_{\infty k}^{\s}$.
    
    Pour $ j'\leq j, k'\leq k$, on a $A_{j'k'}^{\s} = A_{jk}^{\s} \cap A_{j'k'}$, et $A_{00}^{\s} = A$.  
    
  Pour $j\in \N$, on a $A_{j0}^{\s}= K_j^{\s}[[T_{\leq n}^{\frac{1}{p^{j}}}]]$.  Le syst\`eme $(A_{j0}^{\s})_j$ est \`a fl\`eches de transition finies et plates, de sorte que $\, A_{\infty0}^{\s}\,$ est fid\`element plate (en fait libre) sur chaque $\,A_{j0}^{\s}$. D'apr\`es le lemme \ref{L1}, $\, \hat A_{\infty0}^{\s}\,$ est fid\`element plate sur chaque $\,A_{j0}^{\s}$, donc aussi sur $\, A_{\infty0}^{\s}$.  
  
  On a $\, \hat A_{\infty0}^{\s}\cong W(k[[T_{\leq n}^{\e}]])\hat\otimes_{W(k)} \hat K_\infty^{\s}\,$ \cite[ex. 3.2.3 (2)]{A} (nous n'en ferons pas usage).

     \subsection{}  Puisque $A_{\infty k}^{\s}$ est compl\`etement int\'egralement ferm\'ee dans $A_{\infty k}$, et que $\vert K_\infty\vert$ est dense dans $\R_+$, il existe une unique norme de $K_\infty$-alg\`ebre sur $A_{\infty k}$ dont $A_{\infty k}^{\s}$ est la boule unit\'e, et cette norme est spectrale (\ie multiplicative pour les puissances) \cite[sorite 2.3.1 (4)]{A}. 
 Si l'on munit les $A_{jk}$ de la norme (spectrale) induite, la boule unit\'e est $A_{jk}^{\s}$, et lorsque $(j,k)$ varie, les fl\`eches de transition du syst\`eme $(A_{jk}^{\s})$ (les inclusions naturelles) sont isom\'etriques  \cite[sorite 2.3.1 (2c)]{A}.  

 Les $A_{j0}$ sont multiplicativement norm\'ees, mais ce n'est pas n\'ecessairement le cas de $A_{jk}$ si $k\neq 0$. L'id\'eal $\,A_{\infty0}^{\s\s}$ de $\,A_{\infty0}^{\s}$ form\'e des \'el\'ements topologiquement nilpotents est un id\'eal premier idempotent  \'egal \`a $K_\infty^{\s\s}\,A_{\infty0}^{\s}$ \cite[2.2.1, 2.2.2]{A}.

\subsection{}  Nous renvoyons \`a \cite[\S 1]{A} pour un r\'esum\'e des notions de presque-alg\`ebre utilis\'ees dans la suite (notamment \S 1.5  pour les changements de cadre).

  \begin{thm}\label{T3} Dans le cadre $(K_\infty^{\s}, K_\infty^{\s\s}),\; \hat A_{\infty\infty}^{\s } $ est presque fid\`element plate sur $A_{\infty0}^{\s}$.   \end{thm}

     \begin{proof}   Comme $\hat K_\infty$ est un corps perfecto\"{\i}de, il existe $\varpi\in \hat K_\infty^{\s}$ tel que $\vert p\vert \leq \vert \varpi \vert <1$ et admettant des racines $\varpi^{\frac{1}{p^h}}\in \hat K_\infty^{\s}$; on le fixe, ainsi qu'un \'el\'ement  $\varpi'\in K_\infty^{\s}$ tel que  $\vert \varpi\vert < \vert \varpi'\vert <1$. Si $\sB$ est une $\hat K_\infty$-alg\`ebre de Banach, on note $\sB^{\s}$ l'anneau de ses \'el\'ements de puissance born\'ee; on a donc $\sB_{\leq 1}\subset  \sB^{\s}$, avec \'egalit\'e si $\sB$ est spectrale.
     
 \smallskip Commen\c cons par \'etablir quelques formules, exprimant $\hat A_{\infty\infty}^{\s } $ en termes d'une triple colimite compl\'et\'ee.    D'apr\`es \cite[cor. 2.9.2]{A}, on a un isomorphisme canonique 
 \begin{equation}\label{e1}\displaystyle \hat A_{\infty\infty}^{\s}\; \stackrel{\sim}{\to} \;\widehat{\rm{colim}}_{i}\; \hat A_{\infty0}\langle T^{\e}\rangle \{ \frac{T-g}{\varpi^{i}}\}^{\s},\end{equation} 
 o\`u $\hat A_{\infty0}\langle T^{\e}\rangle$ le compl\'et\'e $p$-adique de $ A_{\infty0}[ T^{\e}]$ (c'est aussi le compl\'et\'e pour la norme multiplicative de Gauss), et $\widehat{\rm{colim}}$ le compl\'et\'e $p$-adique de la colimite \cite[\S 2.6.3]{A}; on a  \begin{equation}\label{e2}\hat A_{\infty0}\langle T^{\e}\rangle^{\s} = \hat A_{\infty0}^{\s}\langle T^{\e}\rangle = \hat A_{\infty0}\langle T^{\e}\rangle_{\leq 1}. \end{equation}
et $T-g\,$ est de norme $1$.
   
 \smallskip Nous allons tirer parti de ce que  $\hat A_{\infty0}$ est une $\hat K_\infty$-alg\`ebre perfecto\"{\i}de (c'est celle not\'ee $\hat A_\infty$ dans \cite[ex. 3.2.3 (2)]{A}), donc $\hat A_{\infty0}\langle T^{\e}\rangle $ aussi \cite[ex. 3.2.3 (1)]{A},
  et du fait qu'on dispose d'apr\`es Scholze d'une description en presque-alg\`ebre des boules unit\'e des localisations d'alg\`ebres perfecto\"{\i}des. En effet, selon \cite[cor. 6.7 i]{S1}, il existe un \'el\'ement $f_{i}\in   \hat A_{\infty0}^{\s}\langle T^{\e}\rangle$, congru \`a $T-g$ modulo $\varpi'$ et admettant des racines $f_{i}^{\frac{1}{p^k}}$ dans $\hat A_{\infty0}^{\s}\langle T^{\e}\rangle$, tel que 
   \begin{equation}\label{e3} \hat A_{\infty0}\langle T^{\e}\rangle \{ \frac{T-g}{\varpi^{i}}\}^{\s}\cong \hat A_{\infty0}\langle T^{\e}\rangle \{ \frac{f_i}{\varpi^{i}}\}^{\s}  \end{equation}  
    En outre, selon  \cite[lemma 6.4]{S1}, le morphisme canonique 
    $ \widehat{\hat A_{\infty0}\langle T^{\e}\rangle^{\s}[ (\frac{f_i}{\varpi^{i}})^{\e}]} \to  \hat A_{\infty0}\langle T^{\e}\rangle \{ \frac{f_i}{\varpi^{i}}\}^{\s}   $ est un presque-isomorphisme: 
       \begin{equation}\label{e4} \hat A_{\infty0}\langle T^{\e}\rangle \{ \frac{f_i}{\varpi^{i}}\}^{\s a} \cong (\widehat{\hat A_{\infty0}^{\s}\langle T^{\e}\rangle[ (\frac{f_i}{\varpi^{i}})^{\e}]})^a \end{equation}
  o\`u $\hat A_{\infty0}^{\s}\langle T^{\e}\rangle[ (\frac{f_i}{\varpi^{i}})^{\e}] := {\rm{colim}}_k \, \hat A_{\infty0}^{\s}\langle T^{\e}\rangle[ (\frac{f_i}{\varpi^{i}})^{\frac{1}{p^k}}] \subset   \, \hat A_{\infty0}\langle T^{\e}\rangle$. Par ailleurs, le morphisme canonique 
    $$ \widehat{ {\rm{colim}}}_k \,\hat A_{\infty0}^{\s}\langle T^{\e}\rangle[ (\frac{f_i}{\varpi^{i}})^{\frac{1}{p^k}}]  \to
   \widehat{  {\rm{colim}}}_k \, (\widehat{\hat A_{\infty0}^{\s}\langle T^{\e}\rangle[ (\frac{f_i}{\varpi^{i}})^{\frac{1}{p^k}}]}) $$ 
   est un isomorphisme, comme on le v\'erifie imm\'ediatement par r\'eduction modulo $\varpi^n$ pour tout $n$. En combinant ceci aux formules \eqref{e3} \eqref{e4} et \eqref{m3}, on obtient
   \begin{equation}\label{e6}  \hat A_{\infty0}\langle T^{\e}\rangle \{ \frac{T-g}{\varpi^{i}}\}^{\s a} \cong  \widehat{  {\rm{colim}}}_k \, (\widehat{\hat A_{\infty0}^{\s}\langle T^{\e}\rangle[ (\frac{f_i}{\varpi^{i}})^{\frac{1}{p^k}}]})^a  \end{equation}
   $$\cong   \widehat{  {\rm{colim}}}_k \, (\hat A_{\infty0}^{\s}\langle T^{\e}, U\rangle/(\varpi^{\frac{i}{p^k}} U -  {f_i}^{\frac{1}{p^k}} ))^a .$$
            
         \smallskip Fixons $(i,k)$.    
     Comme $\hat A_{\infty0}^{\s}\langle T^{\e}, U\rangle/(\varpi^{\frac{i}{p^k}} U -  {f_i}^{\frac{1}{p^k}} )$
       ne change pas si l'on remplace $ f_i^{\frac{1}{p^k}}$  par tout \'el\'ement de $ \hat A_{\infty0}^{\s}\langle T^{\e}\rangle $ qui lui est congru modulo $ \varpi^{\frac{i}{p^k}}$, on peut le remplacer par un $f_{ik}\in  A_{j0}^{\s}\langle T^{\frac{1}{p^j}}\rangle $ pour $j$ assez grand. On peut aussi remplacer ${\varpi^{\frac{i}{p^k}}} $ par un \'el\'ement $ \varpi_{ik} \in K_{j0}$ de m\^eme norme, de sorte que
     \begin{equation}\label{e7} \hat A_{\infty0}^{\s}\langle T^{\e}, U\rangle/(\varpi^{\frac{i}{p^k}} U -  {f_i}^{\frac{1}{p^k}} ) \cong  \hat A_{\infty0}^{\s}\langle T^{\e}, U\rangle/({\varpi_{ik}} U -  {f_{ik}} ). \end{equation} 
    Par ailleurs, comme $ \hat A_{\infty0}\langle T^{\e}\rangle $ est perfecto\"{\i}de (spectrale), pour $j$ assez grand, il existe $g_k\in  A_{j0}^{\s}\langle T^{\frac{1}{p^j}}\rangle   $  tel que $g_k^{p^k}\equiv g$ modulo $\varpi$.  On peut aussi supposer qu'une uniformisante $\varpi_j$ de $K_{j0}^{\s} $ v\'erifie \begin{equation}\vert \varpi_j\vert \geq \vert\varpi'\vert^{\frac{1}{p^k}} > \vert \varpi_{ik}\vert.\end{equation}
Le morphisme canonique 
    \begin{equation}\label{e8} \widehat{{\rm{colim}}}_j\,    A_{j0}^{\s}\langle T^{\frac{1}{p^j}}, U\rangle/({\varpi_{ik}} U -  {f_{ik}} ) \to  \hat A_{\infty0}^{\s}\langle T^{\e}, U\rangle/({\varpi_{ik}} U -  {f_{ik}} )  \end{equation} est un isomorphisme, comme on le voit ais\'ement par r\'eduction modulo $\varpi_{ik}^n$ pour tout $n$.
   Enfin, rappelons qu'en vertu de \eqref{m2} et de ce que $A_{j0}^{\s}\langle T^{\frac{1}{p^j}}, U\rangle= A_{j0}\langle T^{\frac{1}{p^j}}, U\rangle_{\leq 1} $, on a 
    \begin{equation}  A_{j0}^{\s}\langle T^{\frac{1}{p^j}}, U\rangle/({\varpi_{ik}} U -  {f_{ik}} ) =  A_{j0}\langle T^{\frac{1}{p^j}}\rangle  \{ \frac{f_{ik}}{\varpi_{ik}}\}_{\leq 1}  . \end{equation}
     
     \medskip Ceci \'etabli, venons-en \`a la platitude. 
     
     Commen\c cons par l'anneau noeth\'erien $  A_{j0 }^{\s}\langle T^{\frac{1}{p^j}}, U\rangle/( \varpi_{ik} U- f_{ik})$. D'apr\`es le crit\`ere de platitude par fibres sur $A_{j0}^{\s}$, il s'agit de v\'erifier que

\smallskip  $a)$ $A_{j0}\langle T^{\frac{1}{p^j}}\rangle  \{ \frac{f_{ik}}{\varpi_{ik}}\}$ est plate sur $A_{j0}$; or, le spectre analytique de $A_{j0}$ est r\'eunion croissante des polydisques affino\"ides de rayons $ r\in \vert K_\infty^{\s\s}\vert$, et pour une telle alg\`ebre de Tate $B_r$,  $B_r\langle T^{\frac{1}{p^j}}\rangle \{ \frac{f_{ik}}{\varpi_{ik}}\}$ est m\^eme plate sur $B_r\langle T^{\frac{1}{p^j}}\rangle$ \cite[prop. 2.2.4]{Be}; et

 \smallskip $b)$ $  A_{j0}^{\s}\langle  T^{\frac{1}{p^j}}, U\rangle/(\varpi_{ik} U-  f_{ik}, \varpi_j) = A_{j0}^{\s}\langle  T^{\frac{1}{p^j}}, U\rangle/( f_{ik}, \varpi_j) $
  est plate sur $A_{j0}^{\s}/\varpi_j$: or dans  
 $A_{j0}^{\s}\langle T^{\frac{1}{p^j}}\rangle$, on a les congruences
$f_{ik}^{p^k} \equiv   f_i  \equiv T-g \equiv (T^{\frac{1}{p^k}}- g_k)^{p^k}$ modulo $\varpi'$, d'o\`u  
$ f_{ik} \equiv T^{\frac{1}{p^k}}- g_k $ modulo $\varpi_j$; ceci donne
$A_{j0}^{\s}\langle  T^{\frac{1}{p^j}}, U\rangle/( f_{ik}, \varpi_j) \cong (A_{j0}^{\s}/\varpi_j)[T^{\frac{1}{p^j}}, U]/  (T^{\frac{1}{p^k}} - g_k), $  qui 
est libre sur $A_{j0}^{\s}/\varpi_j$. 

\smallskip Ainsi $ A_{j0}\langle T^{\frac{1}{p^j}}\rangle  \{ \frac{f_{ik}}{\varpi_{ik}}\}_{\leq 1} $ est plate sur $A_{j0}^{\s}$.  Comme $ A_{j0}^{\s}$ est locale, le quotient par ${\frak m}_{A_{j0}^{\s}}$ de  $A_{j0}^{\s}\langle  T^{\frac{1}{p^j}}, U\rangle/(\varpi_{ik} U-  f_{ik})$ est $k[T^{\frac{1}{p^j}}, U]/ {\bar  f}_{ik}$ o\`u  ${\bar  f}_{ik}$ est l'image de $f_{ik}$ dans $k[ T^{\frac{1}{p^j}}]$, et comme ${\bar  f}_{ik}$ n'est pas inversible dans $k[ T^{\frac{1}{p^j}}]$ puisque ${\bar  f}_{ik}^{p^k} \equiv T-\bar g$, ce quotient est non nul. Donc  {\it $ A_{j0}\langle T^{\frac{1}{p^j}}\rangle  \{ \frac{f_{ik}}{\varpi_{ik}}\}_{\leq 1} $ est fid\`element plate sur $A_{j0}^{\s}$}.

\smallskip
D\`es lors, il en est de m\^eme de  $\,{\rm{colim}}_j A_{j0}\langle T^{\frac{1}{p^j}}\rangle  \{ \frac{f_{ik}}{\varpi_{ik}}\}_{\leq 1}\,$ sur $\,{\rm{colim}}_j A_{j0}^{\s} = A_{\infty 0}^{\s}$, c'est-\`a-dire sur chaque $A_{j0}^{\s}$.  D'apr\`es le lemme \ref{L1}, la colimite compl\'et\'ee est encore fid\`element plate sur $A_{\infty 0}^{\s}$, donc  {\it $ \hat A_{\infty0}^{\s}\langle T^{\e}, U\rangle/(\varpi^{\frac{i}{p^k}} U -  {f_i}^{\frac{1}{p^k}} )$ est fid\`element plate sur $A_{\infty 0}^{\s}$} d'apr\`es \eqref{e7} et \eqref{e8}. 
 
 De m\^eme, la colimite (en $k$) compl\'et\'ee est fid\`element plate sur $A_{\infty 0}^{\s}$. Compte tenu de \eqref{e6}, on obtient que  {\it $ \hat A_{\infty0}\langle T^{\e}\rangle \{ \frac{T-g}{\varpi^{i}}\}^{\s a}$ est fid\`element plate sur $A_{\infty 0}^{\s a}$} dans le cadre $(K_\infty^{\s}, K_\infty^{\s\s})$, ou ce qui revient au m\^eme, dans le cadre $(A_{\infty0}^{\s}, A_{\infty0}^{\s\s}= K_\infty^{\s\s}A_{\infty0}^{\s})$ \cite[\S 1.5]{A}. Dans ce dernier cadre, l'adjoint \`a gauche $( \,)_{!!}$ du foncteur de localisation $(\,)^a$ pour les $A_{\infty0}^{\s}$-alg\`ebres
 %\footnote{donn\'e par $\frak A \mapsto (A_{\infty0}^{\s a}\cdot 1 + K_\infty^{\s\s}\frak A) \subset \frak A$.}
  respecte la platitude fid\`ele \cite[3.1.3]{GR1}, donc $\hat A_{\infty0}\langle T^{\e}\rangle \{ \frac{T-g}{\varpi^{i}}\}^{\s a}_{!!}   $ est fid\`element plate sur $A_{\infty0}^{\s }$. Ces alg\`ebres forment un syst\`eme inductif (en $i$). Appliquant derechef le lemme \ref{L1}, on obtient que la colimite compl\'et\'ee est fid\`element plate sur $A_{\infty0}^{\s}$. En vertu de \eqref{e1}, on conclut que {\it $\hat A_{\infty\infty}^{\s a}$ est fid\`element plate sur $A_{\infty0}^{\s a}$}. \end{proof}

  \subsection{} Bien que cela ne soit pas n\'ecessaire pour la suite, expliquons comment en d\'eduire de la puret\'e:

\begin{cor}\label{C2} Pour tout $(i,j, k)$,  $A_{j0}^{\s}\inj A_{jk}^{\s}$ est pur. 
\end{cor}  

\begin{proof} Il suffit de montrer que $A_{\infty\infty}^{\s}$ est pur sur $A_{\infty0}^{\s}$, puisque ce dernier est pur sur chaque $A_{j0}^{\s}$. 

Comme les $A_{j0}^{\s}$ sont finis plats les uns sur les autres, $A_{\infty0}^{\s}$ est coh\'erent \cite[I, \S 2, ex. 12]{B}.
  Soient alors $M$ un $A_{\infty0}^{\s}$-module de pr\'esentation finie, et $N$ un quelconque sous-module de type fini de $\;{\ker}(M \to M\otimes_{A_{\infty0}^{\s}} A_{\infty\infty}^{\s})$. Comme $A_{\infty0}^{\s}$ est coh\'erent, $N$ est de pr\'esentation finie. Comme $\hat A_{\infty\infty}^{\s}$ est presque fid\`element plat sur $A_{\infty0}^{\s}$ (th. \ref{T3}), donc presque pur, $ A_{\infty\infty}^{\s}$ est aussi presque pur sur $A_{\infty0}^{\s}$, donc $N$ est presque nul: son annulateur $ \frak I$ contient $K_\infty^{\s\s}$. 
  Or $\frak I$ est un id\'eal de pr\'esentation finie, donc provient d'un id\'eal $\frak I_j$ de l'un des $A_{j0}^{\s}: \; \frak I = \frak I_j\otimes_{A_{i,j,0}^{\s}}A_{\infty0}^{\s}$. Comme $A_{\infty0}^{\s}$ est fid\`element plat sur l'anneau local $A_{j0}^{\s}$, on conclut par le lemme \ref{L2} (avec $\frak k=A_{\infty0}^{\s\s}$ et $\frak J = \frak I_j$)  que $\frak I_j=  A_{j0}^{\s}$  et que $ N=0$. Donc $M \to M\otimes_{A_{\infty0}^{\s}} A_{\infty\infty}^{\s}$ est injectif.
 \end{proof}

        \subsubsection{\it Remarque.}\label{re} 
       \smallskip     Si $g$ est sans facteur carr\'e (dans l'anneau factoriel $A[\frac{1}{p}]$) et non divisible par les $T_h$, les $A_{jk}$ et $A_{jk}^{\s}$ sont des anneaux normaux.  
      
      \smallskip    Si $\pm g$ n'est pas produit d'un mon\^ome en les $T_h$ et d'une puissance $p$-i\`eme dans $A$, on peut montrer, en utilisant le lemme de Capelli-Vahlen, que $A_{jk}$ est int\`egre. Si $\pm g$ mod. $p$ n'est pas produit d'un mon\^ome en les $T_h$ et d'une puissance $p$-i\`eme dans $A/p$, on peut montrer que la norme de $A_{jk}$ est multiplicative.
      
       \smallskip    En g\'en\'eral, m\^eme sous les hypoth\`eses de $(1)$ et $(2)$, il est tr\`es difficile de d\'eterminer les anneaux $A_{jk}^{\s}$, et plus encore leurs propri\'et\'es relatives\footnote{il n'est d\'ej\`a pas facile de d\'eterminer $Q(A_{jk})^{\s}$ \`a cause de la ramification f\'eroce \'eventuelle; \cf \cite{W} pour une approche algorithmique - c'est dans cet article oubli\'e qu'est introduite la terminologie {``f\'eroce"} (fierce).}; voir \cite{Ko} pour le cas $j, k\leq 1$ et \cite{R} pour le cas $j\leq 1, k >1$.
       
            \subsubsection{\it Remarque.}\label{re} Dans la suite, nous n'aurons besoin de ce th\'eor\`eme que modulo les puissance de $p$, ce qui permet de se dispenser des subtilit\'es sur les compl\'etions. 
     
 %   \smallskip\noindent $(4)$ Si $g'$ est un autre \'el\'ement de $A$ et $\hat A'_{\infty\infty}$ la $\hat A_{\infty 0}$-alg\`ebre perfecto\"{\i}de associ\'ee comme ci-dessus \`a $g'$, alors $\hat A_{\infty\infty}\hat\otimes_{\hat A_{\infty 0}}\hat A'_{\infty\infty}$ est perfecto\"{\i}de et $(\hat A_{\infty\infty}\hat\otimes_{\hat A_{\infty 0}}\hat A'_{\infty\infty})^{\s}$ est presque isomorphe \`a $\hat A_{\infty\infty}^{\s}\hat\otimes_{{\hat A_{\infty 0}^{\s}}} \hat A_{\infty\infty}^{'\s}$ (\cite[prop. 6.18]{S1}, voir aussi \cite[prop. 3.3.4]{A}), et contient une suite compatible de racines $p^m$-i\`emes de $gg'$. D'apr\`es le th. pr\'ec\'edent,  $\hat A_{\infty\infty}^{\s}$ et $ \hat A_{\infty\infty}^{'\s} $ sont presque fid\`element plates sur $A_{\infty 0}^{\s}$, c'est-\`a-dire sur chaque $A_{jk}^{\s}$, donc $\hat A_{\infty\infty}^{\s}\otimes_{A_{\infty 0}^{\s}} \hat A_{\infty\infty}^{'\s} $ l'est aussi, de m\^eme que son compl\'et\'e d'apr\`es le lemme \ref{L1}. Ainsi $(\hat A_{\infty\infty}\hat\otimes_{\hat A_{\infty 0}}\hat A'_{\infty\infty})^{\s}$ est presque fid\`element plate sur $A_{\infty 0}^{\s}$.

\section{Application du lemme d'Abhyankar perfecto\"{\i}de.}

 \subsection{}\label{lA} Soit $B$ une extension finie de $A$, et prenons $g\in A\setminus pA$ de sorte que $B[\frac{1}{pg}]$ soit \'etale sur $A[\frac{1}{pg}]$. 
  Pour prouver la puret\'e de $A \inj B$, on peut supposer $B$ int\`egre \cite[lemma 3]{H1}, et m\^eme que $B[\frac{1}{pg}]$ soit une extension galoisienne de $A[\frac{1}{pg}]$. Comme $A$ est int\'egralement ferm\'e dans $A[\frac{1}{p}]$, $A/p^m \to B/p^m$ est injectif pour tout $m\in \N$.
 
 La $\hat K_\infty$-alg\`ebre $\hat A_{\infty\infty}$ est perfecto\"{\i}de \cite[\S 3.6.2]{A}. Notons 
 $\hat A_{\infty\infty!!}^{\s }$ la $A_{\infty0}^{\s}$-alg\`ebre d\'eduite de $\hat A_{\infty\infty}^{\s}$ par application de la localisation dans le cadre  $(A_{\infty0}^{\s},   K_\infty^{\s\s}A_{\infty0}^{\s})$ suivie de son adjoint \`a gauche, comme ci-dessus. Comme cette op\'eration respecte la platitude fid\`ele, il d\'ecoule du th. \ref{T3} que {\it $\hat A_{\infty\infty!!}^{\s }$ est fid\`element plate sur $A_{\infty0}^{\s}$ donc aussi sur $A$} (et en particulier sans $A$-torsion).  D'apr\`es la formule (2.2.26) de \cite{GR1}, c'est la sous-alg\`ebre $A_{\infty0}^{\s}+  K_\infty^{\s\s}\hat A_{\infty\infty}^{\s}$ de $\hat A_{\infty\infty}^{\s}$; en particulier, elle est stable par multiplication par tout \'el\'ement de $(\varpi g)^{\e}$.

 \subsection{}\label{Bzero}   Consid\'erons la fermeture int\'egrale $\sB^{\s}$ de $ \, g^{\f}\hat A_{\infty\infty}^{\s} $ dans la $\hat A_{\infty\infty}[\frac{1}{g}]$-alg\`ebre \'etale finie $B\otimes_A \hat A_{\infty\infty}[\frac{1}{g}]$ (galoisienne si l'on veut). 
  Le morphisme canonique $B \otimes_A \hat A_{\infty\infty!!}^{\s } \to B\otimes_A \hat A_{\infty\infty}[\frac{1}{g}]$ se factorise \`a travers un morphisme $B \otimes_A \hat A_{\infty\infty!!}^{\s } \to \sB^{\s}$ (qui devient un isomorphisme apr\`es inversion de $pg$).

 Voyons $\hat A_{\infty\infty}^{\s} $ et $\sB^{\s}$ comme des $K_\infty^{\s}[T^{\e}]$-alg\`ebres via $T^{\frac{1}{p^h}}\mapsto g^{\frac{1}{p^h}}$. Un fragment du lemme d'Abhyankar perfecto\"{\i}de \cite[th. 0.3.1]{A}\footnote{le cas galoisien suffit.} s'\'enonce:

\begin{thm}\label{T4} Dans le cadre $(K_\infty^{\s}[T^{\e}], T^{\e}K_\infty^{\s\s}[T^{\e}])$, et pour tout $m\in \N$, le morphisme $\hat A_{\infty\infty}^{\s }/p^m \to \sB^{\s}/p^m$ est presque fid\`element plat, donc presque pur. \qed\end{thm} 

\begin{cor}\label{C3} Dans le cadre $(K_\infty^{\s}+ K_\infty^{\s\s}[T^{\e}], T^{\e}K_\infty^{\s\s}[T^{\e}])$, et pour tout $m\in \N$, le morphisme $\hat A_{\infty\infty!!}^{\s }/p^m \to (B  \otimes_A \hat A_{\infty\infty!!}^{\s })/p^m$ est presque pur.\end{cor} 

En effet, l'\'enonc\'e du th\'eor\`eme, qui porte sur les $\hat A_{\infty\infty}^{\s }/p^m$-modules, \'equivaut \`a l'\'enonc\'e analogue dans le cadre $(K_\infty^{\s}+ K_\infty^{\s\s}[T^{\e}], T^{\e}K_\infty^{\s\s}[T^{\e}])$ puisque la presque-nullit\'e y a le m\^eme sens; et dans ce cadre-ci, $\hat A_{\infty\infty!!}^{\s }$ est presque isomorphe \`a  $\hat A_{\infty\infty}^{\s }$. Donc $\hat A_{\infty\infty!!}^{\s }/p^m \to \sB^{\s }/p^m$ y est presque pur, de m\^eme que $\hat A_{\infty\infty!!}^{\s }/p^m \to B  \otimes_A \hat A_{\infty\infty!!}^{\s }/p^m $ \`a travers lequel $\hat A_{\infty\infty!!}^{\s }/p^m \to \sB^{\s }/p^m$  se factorise. \qed   

  \subsection{} Fixons provisoirement $m\geq 2$, notons avec une barre la r\'eduction modulo $p^m$ pour all\'eger, et d\'emontrons que la classe $\,e\,$ de l'extension $0\to \bar A \to \bar B \to \bar B/\bar A\to 0$ est nulle en combinant les deux th\'eor\`emes pr\'ec\'edents, suivant une id\'ee de B. Bhatt \cite{Bh}. 
  
  Comme $\overline{\hat A_{\infty\infty!!}^{\s }}$ est (fid\`element) plat sur $\bar A$ et $\bar B$ de pr\'esentation finie sur $\bar A$, on a
$${\rm{Ext}}^1(\bar B/\bar A, \bar A)\otimes_{\bar A} \overline{\hat A_{\infty\infty!!}^{\s }}  = {\rm{Ext}}^1((\bar B  \otimes_{\bar A}  \overline{\hat A_{\infty\infty!!}^{\s }})/  \overline{\hat A_{\infty\infty!!}^{\s }}\,, \, \overline{\hat A_{\infty\infty!!}^{\s }}).$$ Le corollaire \ref{C3} joint au lemme \ref{L10} implique que  $\, e \otimes 1\,$ 
 est annul\'e par $g^{\e}K_\infty^{\s\s}$. 
  Appliquons alors le lemme \ref{L2} avec $$R=\bar A, \;{\frak J} = {\rm{ Ann}}_{\bar A}  \,\bar Ae,\; S=  \overline{\hat A_{\infty\infty!!}^{\s }} , \;\frak K= g^{\e}K_\infty^{\s\s} \overline{\hat A_{\infty\infty!!}^{\s }}.$$  Notons que ${\frak J} S = {\rm{ Ann}}_{ \overline{\hat A_{\infty\infty!!}^{\s }}}  \, \overline{\hat A_{\infty\infty!!}^{\s }} (e\otimes 1)$ puisque $ \overline{\hat A_{\infty\infty!!}^{\s }}$ est plat sur $\bar A$, de sorte que ${\frak J} S$ contient $\frak K$.
 Puisque $R\cap \frak K \neq 0$ (il contient la classe de $pg$), on conclut que ${\frak J}= R$, c'est-\`a-dire $\, e = 0$. 
 
  \subsection{} On a donc obtenu l'existence d'une r\'etraction de $A/p^m\to B/p^m $ pour tout $m$. Il en est donc de m\^eme pour $  A/(p^m, T_{\leq n}^{p^m}) \to   B/(p^m, T_{\leq n}^{p^m})$. Un argument de type Mittag-Leffler d\^u \`a Hochster \cite[p. 30]{H1}, bas\'e sur le fait que ces r\'etractions forment un torseur sous un $A/p^m$-module artinien, permet de conclure que $A\to B$ admet une r\'etraction.

  \section{Alg\`ebres de Cohen-Macaulay.}

 \subsection{}\label{mp} Soient $B$ un anneau local noeth\'erien de caract\'eristique r\'esiduelle $p$, $\frak m_B$ son id\'eal maximal, et $\underline{x} = (x_1, \ldots, x_d)$ une suite s\'ecante maximale ($d= \dim B$).   Soit $C$ une extension non n\'ecessairement noeth\'erienne de $B$. 
  
  \begin{defn} \begin{enumerate} \item On dit que $C$ est {\emph{de Cohen-Macaulay pour $(B,\underline{x} )$}} si $C \neq \frak m_B C$ et si $\underline{x}$ devient une suite r\'eguli\`ere dans $C$.  
  
  \item On dit que $C$ est une {\emph{$B$-alg\`ebre de Cohen-Macaulay}}\footnote{``big Cohen-Macaulay algebra" ou ``balanced big Cohen-Macaulay algebra" dans la litt\'erature anglo-saxonne.} si elle est de Cohen-Macaulay pour $(B,\underline{x} )$ pour toute suite s\'ecante maximale $\underline{x}$. 
  
  \item Supposons que $C$ contienne une suite compatible de racines $p^m$-i\`emes d'un \'el\'ement $\pi$ non diviseur de z\'ero dans $C$. Suivant P. Roberts, on dit que $C$ est {\emph{presque de Cohen-Macaulay pour $(B,\underline{x}, \pi^{\e})$}} si dans le cadre $(C, \pi^{\e} C)$, $C$ n'est pas presque \'egale \`a $\frak m_B C$ (\ie $\frak m_B C$ ne contient pas  $\pi^{\e}$) et si $\underline{x}$ devient presque une suite r\'eguli\`ere dans $C$ (\ie pour tout $i= 0, \ldots, d-1$, $((x_1, \ldots, x_i)C:  x_{i+1}C)/ (x_1, \ldots, x_i)C $ est annul\'e par $\pi^{\e}$).   
  \end{enumerate}  \end{defn} 
  
 On {``passe "} de $(1)$ \`a $(2)$ en compl\'etant $C$ $\frak m_B$-adiquement \cite[th. 1.7]{BS},
 %\cite[cor. 8.5.3]{BH},
  et de $(3)$ \`a $(1)$ gr\^ace \`a la technique des modifications partielles de Hochster \cite{H3}, qui peut se r\'esumer comme suit. Une {\it modification partielle de degr\'e $n$ d'un $B$-module $M$ relatif \`a $(B, \underline{x})$} est un homomorphisme $M\to M'$, o\`u, \'etant donn\'ee une relation $x_{i+1}m_{i+1} = \sum_1^i x_j m_j$ \`a coefficients dans $M$,  $M': = M[T_1, \ldots, T_i]_{\leq n}/ (m_{i+1} - \sum_1^i x_j T_j)\cdot M[T_1, \ldots, T_i]_{\leq n -1} . $

   \begin{prop}\label{P1} Soit $B $ un anneau local noeth\'erien de caract\'eristique r\'esiduelle $p$. Si $B$ admet une alg\`ebre presque de Cohen-Macaulay $C$ pour $(B,\underline{x}, \pi^{\e})$ (pour une suite s\'ecante maximale $\underline{x}$ de $B$ et une suite de racines $\pi^{\frac{1}{p^m}}$ d'un non-diviseur de z\'ero dans $C$), alors $B$ admet une alg\`ebre de Cohen-Macaulay.  
  \end{prop} 
  
   \begin{proof} \cite{H3} Il est loisible de remplacer $C$ par $\pi^{\f}C$, qui est encore presque de Cohen-Macaulay pour $(B,\underline{x}, \pi^{\e})$. Consid\'erons une suite finie $\underline{M} = (M_1 := B \to M_2 \to \cdots \to M_\ell)$ de modifications partielles de degr\'e $n$ relatives \`a $(B, \underline{x})$.  Le lemme crucial \cite[5.1]{H3} (o\`u l'on prend $c = \pi^{\frac{1}{p^m}}$ pour $m$ arbitraire) permet de construire pas \`a pas un diagramme commutatif de $B$-modules, partant de $M_1 = B \to C$:
    \begin{displaymath} 
    \xymatrix @-1.5pc {
    {M_1}  \ar[d] \ar[rdd]       & \\ 
     {M_2}  \ar[d] \ar[rd]    &         & \\
      &    C   ,  \\ 
       \ar[d] &   &\\  
      {M_\ell}  . \ar[ruu] &  & } 
    \end{displaymath} 
Comme   $C\neq \frak m_B C$, on a $M_\ell \neq \frak m_B M_\ell$. Lorsque $(n,\ell,  \underline{M})$ varie, la colimite  (filtrante) des $M_\ell$ est 
 alors une alg\`ebre de Cohen-Macaulay pour $(B,\underline{x} )$. Passant au compl\'et\'e $\frak m_B$- adique, on obtient une alg\`ebre de Cohen-Macaulay pour $B$. \end{proof}

  Voici un moyen commode pour construire des alg\`ebres presque de Cohen-Macaulay. Soient $ B, \underline{x}, C$ comme au d\'ebut du \S \ref{mp}, et supposons que $C$ contienne une suite compatible de racines $p^m$-i\`emes d'un \'el\'ement $\pi$ non diviseur de z\'ero. 
  
   \begin{lemma}\label{L4}  
       Supposons que les $x_i$ soient contenus dans un sous-anneau $A \subset B$, local de Cohen-Macaulay, tel que $B$ soit un $A$-module fini. 
   Alors $C$ est presque de Cohen-Macaulay pour $(B,\underline{x}, \pi^{\e})$ si et seulement si elle l'est pour $(A,\underline{x}, \pi^{\e})$. C'est le cas si l'une des conditions suivantes est v\'erifi\'ee:
   
  $a)$ $(\pi ) \cap  A\neq 0$ et $ C$ est presque isomorphe, dans le cadre $(A[\pi^{\e}], \pi^{\e}A[\pi^{\e}])$, \`a une $A$-alg\`ebre fid\`element plate (ne contenant pas n\'ecessairement $B$).
 
 $b)$ $(\pi) \cap  A \not\subset x_1 A, \;C$ est sans $x_1$-torsion, et $C/x_1 C$ est presque isomorphe, dans le cadre $(A[\pi^{\e}], \pi^{\e} A[\pi^{\e}])$, \`a une $A/x_1A$-alg\`ebre fid\`element plate.   
   \end{lemma}
  
  \begin{proof}   Comme $\frak m_A  B$ est $\frak m_B$-primaire, $C$ n'est pas presque \'egale \`a $\frak m_B C$ si et seulement si elle n'est pas presque \'egale \`a $\frak m_A C$, d'o\`u la premi\`ere assertion. C'est le cas si $C$ (\resp $C/x_1C)$ est presque isomorphe \`a une $A$-alg\`ebre (\resp $A/x_1A$-alg\`ebre) fid\`element plate $C'$ car sous les hypoth\`eses en vigueur, l'image de $ \pi^{\e}$ dans $C'$ ne peut \^etre contenu dans $\frak m_A C'$ en vertu du lemme \ref{L2}. 
  
  Par ailleurs, $((x_1, \ldots, x_i)A:  x_{i+1}A)/ (x_1, \ldots, x_i)A  =0$ puisque $A$ est de Cohen-Macaulay, donc $((x_1, \ldots, x_i)C':  x_{i+1}C')/ (x_1, \ldots, x_i)C'  =0$ et $((x_1, \ldots, x_i)C:  x_{i+1}C)/ (x_1, \ldots, x_i)C  =0$ est annul\'e par  $ \pi^{\e}$.
  On conclut de m\^eme dans le cas $(b)$, rempla\c cant $A$ par son quotient de Cohen-Macaulay $A/x_1A$ et $(x_1, \ldots, x_i)$ par $(x_2, \ldots, x_i)$. 
   \end{proof}

  \subsection{} Venons-en au th. \ref{T2}. Supposons d\'esormais en outre que $B$ soit complet int\`egre de car. $(0, p)$. Soient $n+1$ sa dimension,  $k$ son corps r\'esiduel et $\Lambda\subset W(k^{\e})$ son anneau de coefficients. Le choix d'une suite s\'ecante maximale $\underline{x}$ de $B$ avec $x_1=p^2$ permet d'\'ecrire $B$ comme extension finie de $A := \Lambda[[T_1, \ldots, T_n]]$, o\`u $T_i$ s'envoie sur $x_{i+1}$. Il est loisible de remplacer $k$ par $k^{\e}$ et $\Lambda$ par $W(k^{\e})$, ce qui nous place dans la situation du \S \ref{lA}. 
  
\medskip    Soit $g\in A\setminus pA$ tel que $B[\frac{1}{pg}]$ soit \'etale sur $A[\frac{1}{pg}]$.  Avec les notations de \ref{Bzero}, prenons
  \begin{equation}  C := \sB^{\s},\,  \pi^{\frac{1}{p^m}}:= (\varpi g)^{\frac{1}{p^m}}.\end{equation}
   Alors $(\pi) \cap A = pgA \not\subset p^2 A$, la $\hat A_{\infty\infty !!}^{\s }$-alg\`ebre $C  $ est sans $p$-torsion, et $C/p^2C$ est presque isomorphe, dans le cadre $(\hat A_{\infty\infty !!}^{\s }, (\varpi g)^\e\hat A_{\infty\infty !!}^{\s })$, \`a  $(C/p^2C)^a_{!!}$ qui est fid\`element plate sur $A/p^2 A$ par les th. \ref{T3} et \ref{T4}. On conclut du lemme \ref{L4} que $C$ est presque de Cohen-Macaulay pour $(B, \underline{x} , \pi^{\e})$. Par la prop.  \ref{P1}, il existe donc une $B$-alg\`ebre  de Cohen-Macaulay.   \qed

   \subsubsection{\it Remarque.} Le th. \ref{T2} implique directement la conjecture monomiale (point $(3)$ du \S \ref{fd}), qui est \'equivalente \`a la conjecture du facteur direct. Cela en fournit une seconde preuve, o\`u le lemme d'Abhyankar perfecto\"{\i}de n'intervient que modulo $p^2$. 
   
  \smallskip La conjecture qui contr\^ole tout l'\'echeveau des conjectures homologiques pr\'edit la fonctorialit\'e faible des alg\`ebres de Cohen-Macaulay. Les techniques pr\'ec\'edentes devraient permettre de traiter le cas d'un   homomorphisme local {\it injectif} d'anneaux locaux noeth\'eriens complets int\`egres d'in\'egale caract\'eristique $(0,p)$, 
   mais le cas g\'en\'eral semble requ\'erir une nouvelle id\'ee.

 \bigskip
     \section{Appendice: puret\'e.} 
    
    Comme la conjecture du facteur direct est un \'enonc\'e de puret\'e (au sens d'injectivit\'e universelle), nous rassemblons ici quelques r\'esultats  concernant cette notion. 
 
 \subsection{Sous-modules purs et modules g\'en\'erateurs.} Soient $R$ un anneau commutatif unitaire, $N$ un $R$-module.   
 On dit qu'un sous-module $M\subset N$ est {\it pur} si pour tout $R$-module $P$, $P\to P\otimes_R S$ est injectif. Comme tout module est colimite filtrante de modules de pr\'esentation finie, on peut se borner aux modules $P$ de pr\'esentation finie.
 
  \begin{lemma}\label{L5}\cite[I.2]{L}\begin{enumerate} \item $M\subset  N$ est pur si et seulement si pour tout module de pr\'esentation finie $P$, $Hom_R(P, N)\to Hom_R(P, N/M)$ est surjectif.  
 
 \item En particulier, si $N/M$ est de pr\'esentation finie, $M\subset N$ est pur si et seulement si $M$ est facteur direct de $N$. 
 
 \item Si $N$ est plat, alors $M$ est un sous-module pur si et seulement si $N/M$ est plat. \qed
 \end{enumerate} 
 \end{lemma} 
 
 Un $R$-module $M$ est {\it g\'en\'erateur} si tout $R$-module $N$ est engendr\'e par les images des applications $R$-lin\'eaires de $M$ dans $N$.  
 
 \begin{lemma}\label{L7}\cite[\S 5, n. 2, Th. 1]{B2} $M$ est g\'en\'erateur si et seulement si $R$ est facteur direct d'une puissance $M^n$ (en particulier $M$ est fid\`ele). \qed \end{lemma}

   \subsection{Extensions pures d'anneaux.} On dit qu'un monomorphisme d'anneaux $R\inj S$ est {\it pur} (ou encore que la $R$-alg\`ebre fid\`ele $S$ est {\it pure}) si $R$ est un sous-$R$-module pur de $S$: autrement dit, pour tout $R$-module (de pr\'esentation finie) $P$, $P\to P\otimes_R S$ est injectif. En passant par l'alg\`ebre sym\'etrique sur $P$, on voit $R\inj S$ est pur si et seulement si il est {\it universellement injectif}, \ie $R'\to S\otimes_R R'$ est injectif pour toute $R$-alg\`ebre $R'$ (et on peut se limiter aux $R'$ de pr\'esentation finie puisque toute $R$-alg\`ebre est colimite filtrante de telles alg\`ebres). 
   
  \begin{prop}\label{P2} \begin{enumerate} 
   \item Les monomorphismes purs sont stables par composition, produit, changement de base, et colimite filtrante. En outre, si un compos\'e $R\to S\to T$ est pur, il en est de m\^eme de $R\to S$.
    \item  Si $R\inj S$ est fid\`element plat, il est pur.  
\noindent \item Supposons que $S$ soit de pr\'esentation finie en tant que $R$-module.  Les conditions suivantes sont \'equivalentes:
    \begin{enumerate} \item $R\inj S$  est pur,
      \item $R\inj S$ admet une r\'etraction $R$-lin\'eaire, \ie le sous-$R$-module $R\subset S$ est facteur direct. 
    \item $S$ est un $R$-module g\'en\'erateur,
   \end{enumerate} 
   \end{enumerate}  
  \end{prop} 

Le point $(1)$ est formel, \cf \cite[prop 1.2]{Ol}.  Le point $(2)$ d\'ecoule du point $(3)$ du lemme  \ref{L5}, compte tenu de ce que $R\inj S$ est fid\`element plat si et seulement si $S/R$ est plat.  
 Les implications $(a) \Rightarrow (b)  \Rightarrow (c)  $ de $(3)$ d\'ecoulent en effet directement des lemmes pr\'ec\'edents, et $(b)\Rightarrow (a)$ est banale. Pour $(c)  \Rightarrow (b)$, noter que pour tout couple $(s, \check s)\in S\times Hom_S(S, R)$ tel que $\check s (s) = 1_R$, la composition de la multiplication par $s$ dans $S$ et de $\check s$ est une r\'etraction de $R\inj S$.    \qed
 
   \subsection{} {\it Remarque.}  Notons $f: {\Spec} S \to {\Spec} R$ le morphisme associ\'e \`a $R\to S$ et $f^\ast:\, Mod_R\to Mod_S$ le foncteur de changement de base. Alors $f$ est plat si et seulement si $f^\ast$ est exact, tandis que $f$ est pur si et seulement si $f^\ast$ est fid\`ele (ce qui \'equivaut ici \`a conservatif, \ie refl\'etant les isomorphismes), \cf \cite{Ol}.

    \subsection{Un crit\`ere de puret\'e de Hochster.}\label{astH} Voici une variante de \cite[6.3 et 6.1 (2$\Rightarrow$ 5)]{H2}.
     
     \begin{prop}\label{P3} Soient $R$ un anneau local (non n\'ecessairement noeth\'erien) d'id\'eal maximal $\frak m$, et $r$ un \'el\'ement de $ \frak m$ tel que $R$ soit $r$-adiquement s\'epar\'e. 
     
     Soit $\sigma$ un endomorphisme local de $R$ tel que $R$ soit libre sur $\sigma(R)$ et que $\displaystyle \bigcap_m \, (\sigma^m (\frak m). R)$ soit contenu dans  $r R $. Soit enfin $S$ une extension de $R$ telle que $\sigma$ se prolonge en un endomorphisme injectif de $S$. 
     
     Consid\'erons les conditions
     
  \smallskip  $(a)$   $R\inj S$ est pur,

  \smallskip  $(b)$  $R\inj S$ admet une r\'etraction $R$-lin\'eaire,

 \smallskip    $(c)$ Le dual $S^\vee := Hom_R(S, R)$ est non nul.
  
  \medskip\noindent On a les implications
     \[ (c) \Leftrightarrow (b) \Rightarrow (a).\]
   En outre $(a) \Rightarrow (b)$ si $R$ est (noeth\'erien) r\'egulier et $S$ entier sur $R$.
 \end{prop} 
 
 \begin{proof} L'implication $(b)\Rightarrow (a)+(c)$ est triviale. 
 
 Prouvons $(c)\Rightarrow (b)$. Soit $\lambda \in S^\vee \setminus \{0\}$. Comme $R$ est $r$-adiquement s\'epar\'e, on peut, en divisant $\lambda$ par une puissance convenable de $r$, supposer qu'il existe $s\in S$ tel que $\lambda(s)\notin r R$; quitte \`a pr\'ecomposer $\lambda$ avec la multiplication par $s$, on peut m\^eme supposer $\lambda(1)\notin r R$. Il existe alors $m$ tel que $\lambda(1)\notin \sigma^m(\frak m ) R$. Comme $R$ est libre sur $\sigma^m R$ (qui est un anneau local d'id\'eal maximal $\sigma^m (\frak m)$), il existe un facteur direct (libre) de type fini $M$ tel que $\lambda(1)\in M\setminus \sigma^m (\frak m) M$, et on peut donc trouver d'apr\`es Nakayama une forme $\sigma^m R$-lin\'eaire $\mu$ sur $N$ qui envoie $\lambda(1)$ sur $1$, qu'on prolonge \`a $R$ par $0$ sur un suppl\'ementaire de $M$. La restriction \`a $\sigma^m S$ de $\mu\lambda$ est alors une r\'etraction $\sigma^m R$-lin\'eaire de $\sigma^m S$ sur $\sigma^m R$. Par transport de structure via $\sigma^m$, on obtient une r\'etraction $R$-lin\'eaire de $ S$ sur $  R$.
 
 Prouvons ensuite $(a) \Rightarrow (c)$ si $R$ est r\'egulier et $S$ entier sur $R$. $S$ est alors colimite filtrante de sous-$R$-alg\`ebres finies $S_\alpha$ qui sont pures. Si $d$ est la dimension de Krull de $R$, le groupe de cohomologie locale $H^d_{\frak m}(R)$ est non nul, et s'injecte dans $H^d_{\frak m}(S_\alpha)$ puisque $R\inj S_\alpha$ est scind\'e. Par passage \`a la colimite, $H^d_{\frak m}(S)$ est non nul. Soit alors $E$ l'enveloppe injective du corps r\'esiduel de $R$. Par dualit\'e locale, $H^d_{\frak m}(S)$ s'identifie \`a $Hom_R(S^\vee, E)$, donc $S^\vee$ est non nul.  
 \end{proof}
 
 L'implication $(c)\Rightarrow (b)$ fournit une preuve tr\`es courte de la conjecture du facteur direct en caract\'eristique $p>0$ \cite[6.2]{H2}, en prenant $r=0$ et  $\sigma$ \'egal \`a l'endomorphisme de Frobenius, qui est plat si $R$ est r\'egulier et fini si $R$ est local complet de corps r\'esiduel parfait.

       \subsection{Presque-puret\'e.}  Soit $(\frak V, \frak m= \frak m^2)$ un cadre tel que $\tilde{\frak m }:= \frak m\otimes_{\frak V} \frak m$ soit plat sur $\frak V$. Soient $R$ une $\frak V$-alg\`ebre, et $N$ un $R$-module. L'adjoint \`a gauche $(\,)_! = \tilde{\frak m }\otimes (\,)_\ast$ de la localisation $(\,)^a$ est exact et commute \`a $\otimes$ \cite[2.2.24, 2.4.35]{GR1}. 
      
      On dit qu'un homomorphisme presque injectif $M\to N$ de $R$-modules est {\it presque pur} si pour tout $R$-module $P$, $P\to P\otimes_R S$ est presque injectif (ici encore, il suffit de tester sur les modules $P$ de pr\'esentation finie). Cela \'equivaut \`a dire que $M^a_!\subset N^a_!$ est pur. Si $M'$ est un module interm\'ediaire (avec $M'\to M$ presque injectif), $M\to M'$ est encore presque pur. 
      
      \begin{lemma}\label{L10} Si $M\subset   N$ est presque pur, alors pour tout $R$-module de pr\'esentation finie $P$, $Hom_R(P, N)\to Hom_R(P, N/M)$ est presque surjectif. En particulier, si $N/M$ est de pr\'esentation finie sur l'anneau $R$ (au sens usuel),  $M$ est alors presque facteur direct de $N$, et donc la classe de $N$ dans le $R$-module $Ext^1(N/M, M)$ est presque nulle. 
      \end{lemma} 
         
         \begin{proof} D'apr\`es le lemme \ref{L5},  $M\subset N$ est presque pur si et seulement si pour tout $R$-module de pr\'esentation finie $P$, $Hom_R(P, N^a_!)\to Hom_R(P, N^a_!/M^a_!)$ est surjectif,
            donc $Hom_R(P, N)\to Hom_R(P, N/M)$ est presque surjectif. Si $N/M$ est de pr\'esentation finie, on peut prendre $P= N/M$, d'o\`u l'existence, pour tout $\eta\in \frak m$, d'un \'el\'ement $f\in  Hom_R(N/M, N)$ dont la projection dans $End_R(N/M)$ est $\eta\cdot id$. 
         \end{proof} 
 
 Un homomorphisme de $\frak V$-alg\`ebres  $R\to S$ est {\it presque pur} s'il l'est en tant qu'homormorphisme de $R$-modules. 
 Si un compos\'e $R\to S\to T$ est presque pur, il en est de m\^eme de $R\to S$.

        \end{sloppypar}

 \end{document}